\documentclass[preprint,ECP]{ejpecp} 

\usepackage{amsfonts}
\usepackage{amsmath}
\usepackage{amssymb}
\usepackage{amsthm}
\usepackage{amsbsy}
\usepackage{graphicx} \usepackage{enumerate}
\usepackage{mathrsfs} \usepackage[all,cmtip]{xy}
\usepackage[color=blue!20]{todonotes}
\usepackage{subfigure}

\numberwithin{equation}{section}

\SHORTTITLE{Most probable flows for Kunita SDEs}

\TITLE{Most probable flows for Kunita SDEs\support{The first author is supported by the grant GeoProCo from the Trond Mohn Foundation - Grant TMS2021STG02 (GeoProCo). The second author is supported by the Villum Foundation Grants 22924 and 40582, and the Novo Nordisk Foundation grant NNF18OC0052000.}
}

\AUTHORS{%
  Erlend~Grong\footnote{University of Bergen, Norway.
    \EMAIL{Erlend.Grong@uib.no}}
  \and 
  Stefan~Sommer\footnote{University of Copenhagen,
    Denmark. \BEMAIL{sommer@di.ku.dk}}}


\AMSSUBJ{62R30, 60D05, 53C17} 

\VOLUME{0}
\YEAR{2020}
\PAPERNUM{0}
\DOI{10.1214/YY-TN}

\ABSTRACT{
We identify most probable flows for Kunita Brownian motions, i.e. stochastic flows with Eulerian noise and deterministic drifts. Such stochastic processes appear for example in fluid dynamics and shape analysis modelling coarse scale deterministic dynamics together with fine-grained noise. We treat this infinite dimensional problem by equipping the underlying domain with a Riemannian metric originating from the noise. The resulting most probable flows are compared with the non-perturbed deterministic flow, both analytically and experimentally by integrating the equations with various choice of noise structures.
}

\newcommand{\calX}{\mathcal{X}}

\newcommand{\scrF}{\mathscr{F}}

\DeclareMathOperator{\id}{id}

\DeclareMathOperator{\spn}{span}
\DeclareMathOperator{\tr}{tr}

\DeclareMathOperator{\Ad}{Ad}
\DeclareMathOperator{\ad}{ad}

\DeclareMathOperator{\Diff}{Diff}

\newcommand{\ve}{\varepsilon}
\DeclareMathOperator{\DIV}{div}

\begin{document}

\section{Introduction}
Stochastic flows appear in a range of applications from fluid dynamics \cite{holmVariationalPrinciplesStochastic2015,ACC14} where small scale, turbulent dynamics can be modelled stochastically to shape modelling \cite{arnaudonGeometricFrameworkStochastic2019} where the shape of human organs or changes of animal morphology through evolution can be modelled as stochastic flows.
We consider Brownian flows of diffeomorphisms in the sense of Kunita \cite{kunitaStochasticFlowsStochastic1997} given by stochastic differential equations of the form
\begin{equation} \label{ShapeSDE} d\varphi_t = u_t(\varphi_t) \, dt + \sum_{r=1}^J \sigma_{r,t}(\varphi_t) \circ dW^r,\qquad \varphi_0=\varphi.\end{equation}
Here $u_t$, $\sigma_{1,t}$, $\dots$, $\sigma_{J,t}$ are time-dependent vector fields on an open domain~$M$ in~$\mathbb{R}^d$, which are smooth in both the time and the space variable. Assume now that we are given a diffeomorphic transformation $\psi$ of $M$. We can consider this a realization of the flow at some positive time $T>0$, and we wish to find the most probable way for a diffeomorphism $\varphi_0=\varphi$ to transform such that $\varphi_T=\psi$.
While being an infinite dimensional problem, we show how the most probable transformations can be identified by using a Riemannian metric on $M$ originating from the noise fields $\sigma_{r,t}$. We use this to identify evolution equations for most probable trajectories point-wise. This in turn gives equations for most probable flows between $\varphi$ and $\psi$.

\subsection{Background}
Most probable paths for Euclidean Brownian motions were identified by Onsager and Machlup \cite{onsagerFluctuationsIrreversibleProcesses1953,machlupFluctuationsIrreversibleProcess1953} and described by the Onsager-Machlup functional $\gamma \mapsto \int_0^T H(\gamma(t),\dot{\gamma}(t)) dt=\frac12 \int_0^T \|\dot{\gamma}(t)\|^2 \, dt$ that measures the probability that realizations of a Brownian motion $W_t$ sojourn around smooth paths in the sense of staying in $\varepsilon>0$ diameter cylinders. 
If instead $\gamma$ is a curve on a finite dimensional Riemannian manifold, the Onsager-Machlup functional changes so that the integral is now over $H(\gamma(t),\dot{\gamma}(t))=\left( \frac12\|\dot{\gamma}(t)\|_g^2-\frac1{12}S(\gamma(t))\right)$ where $S$ is the scalar curvature. Most probable paths on manifolds thus minimize a functional that in addition to the path energy includes the integral of the scalar curvature along the path.

The situation for stochastic processes on infinite dimensional spaces remains largely unexplored. In this paper, we consider Kunita flows of diffeomorphisms of the domain $M\subseteq\mathbb R^d$. For a stochastic flow $\varphi_t\in\Diff(M)$, we ask what is the most probable single realization $\hat{\varphi}_t$ that bridges $\varphi_0$ to a fixed transform $\psi\in\Diff(M)$. We achieve this by equipping $M$ with a Riemannian metric originating from the noise fields $\sigma_{i,t}$. We then invoke the Onsager-Machlup result \cite{CP14} for elliptic time-dependent diffusion operators on $M$ to give a point-wise result. The resulting second-order ODE on $M$ gives the most probable flow $\hat{\varphi}_t$.

Onsager-Machlup functionals for infinite dimensional systems have been studied in the literature
\cite{demboOnsagerMachlupFunctionalsMaximum1991,mayerwolfOnsagerMachlupFunctionals1993,BRT03} for specific Hilbert-space valued SDEs. However, these are not available when considering smooth diffeomorphisms, which form a Lie-Fr\'echet group, see e.g. \cite{Neeb06}.
While the idea of encoding properties of the diffusion in a Riemannian metric has been used in the finite dimensional setting \cite{takahashiProbabilityFunctionalsOnsagermachlup1981a,capitaineOnsagerMachlupFunctionalElliptic2000}, the main new idea of the paper is to use such a metric for infinite dimensional Kunita SDEs.

Particular examples of flows to which our result apply are perturbations of geodesic flows for right-invariant metrics on $\Diff(M)$. These geodesics appear both in fluid dynamics describing e.g. the incompressible Euler equations \cite{arnoldMathematicalMethodsClassical1978} and in shape analysis in the Large deformation diffeomorphic metric mapping (LDDMM) framework \cite{younesShapesDiffeomorphisms2010,bauerOverviewGeometriesShape2014}. 
The stochastically perturbed flows have recently been applied in fluid dynamics to couple small-scale stochastic fluctuations and coarse scale deterministic dynamics 
\cite{holmVariationalPrinciplesStochastic2015,holmStochasticVariationalFormulation2020} 
and in shape analysis to model human organ shape changes or animal shapes changes through evolution
\cite{arnaudonGeometricFrameworkStochastic2019}. By identifying the most probable paths for such processes, we are able to give one summary statistic of a fluid system or provide a most probable evolution of a given species shape through time. Shapes can here be interpreted generally as objects on which $\Diff(M)$ act, e.g. on transformation landmark configurations $x=(x_1,\ldots,x_n)$ where they act by $\varphi.x=(\varphi(x_1),\ldots,\varphi(x_n))$ or on images $I:M\to\mathbb R$ with action $\varphi.I=I\circ\varphi^{-1}$.

\subsection{Outline}
We start by a brief outline of Kunita type stochastic flows, stochastic flows in fluid dynamics and shape analysis, and most probable path theory. We then develop the general construction for retrieving Onsager-Machlup results of flows from the corresponding theory of manifold-valued diffusion processes. We discuss various instances of stochastic models with application in e.g. shape analysis and investigate the case of evolutions of finite sets of landmarks under stochastic shape models together with numerical visualization of the different types of evolutions.

\section{Stochastic flows and measures on path space}
\subsection{Stochastic flows} Given a probability space $(\Omega,\scrF,P)$, let $\varphi_{s,t}$ denote $s,t$ parametrized maps from $M\times\Omega$ to $M$. A Brownian flow of homeomorphism in the sense of Kunita is a stochastic process $\varphi_{s,t}$ such that $\varphi_{s,t}(x,\cdot)$ is continuous in $s,t,x$ in probability and $\varphi_{s,t}(x,\omega)$ is continuous in $t$ almost surely (a.s.); $\varphi_{s,s}=\id_M$ for each $s$ a.s.; $\varphi_{s,t}=\varphi_{r,t}\circ\varphi_{s,r}$; and $\varphi_{s,t}$ has independent increments \cite{kunitaStochasticFlowsStochastic1997}. Particular examples of Brownian flows are solutions $\varphi_t=\varphi_{0,t}$ to the It\^o SDE
\begin{equation} \label{Ito} d\varphi_t = \hat u_t(\varphi_t) \, dt + \sum_{r=1}^J \sigma_{r,t}(\varphi_t)  dW^r\end{equation}
for time-dependent vector fields $\hat u_t,\sigma_{1,t},\ldots,\sigma_{J,t}:M\to\mathbb R^d$. $\hat u$ and $a_t=\sum_{r=1}^J\sigma_{r,t}\sigma_{r,t}^T$ defines the local characteristics of the flow. If the local characteristics are $C^1$ with bounded derivatives and if the derivatives of $a$ are Lipschitz, the It\^o to Stratonovich conversion applies giving the equivalent Stratonovich form \eqref{ShapeSDE}, \cite{kunitaLecturesStochasticFlows1986, kunitaStochasticFlowsStochastic1997}.

\subsection{Path probability and Onsager-Machlup functionals}
Let $\gamma\in C([0,T],\mathbb R^d)$ be a path in $\mathbb R^d$ and $W_t$ a Brownian motion on $\mathbb R^d$. The probability that $W_t$ sojourns in an $\varepsilon$-neighborhood around $\gamma$ for all $t\in [0,T]$ is defined by
\begin{equation}
  \mu_\ve(\gamma)
  =
  P(\|W_t-\gamma(t)\|<\ve\ \forall t\in[0,T])
  .
\end{equation}
In \cite{fujitaOnsagerMachlupFunctionDiffusion1982}, it was shown that, when $\gamma$ is differentiable, $\log \mu_\ve(\gamma)$ tends to $c_1+c_2/\ve^2- \int_0^T H(\gamma(t),\dot{\gamma}(t))\, dt$ for constants $c_1,c_2$ as $\ve\to 0$ and with $H(\gamma(t),\dot{\gamma}(t))=\frac12  \|\dot{\gamma}(t)\|^2$. The function $\gamma \mapsto \int_0^T H(\gamma(t),\dot{\gamma}(t))\, dt$ is called the Onsager-Machlup functional. Paths between two points minimizing this functional are denoted \emph{most probable}, and, because of the equivalence to the standard Euclidean energy of a curve, the most probable paths will be straight lines. If $\gamma$ is a curve on a finite dimensional Riemannian manifold, the Onsager-Machlup functional is the integral of $H(\gamma(t),\dot{\gamma}(t)) = \frac{1}{2} \|\dot{\gamma}(t)\|_g^2-\frac1{12}S(\gamma(t)) \, dt$ with $S$ the scalar curvature.

The functional has been explored in the literature after the original work by Onsager and Machlup. 
Encoding properties of the diffusion processes in a Riemannian metric has been used in works including \cite{takahashiProbabilityFunctionalsOnsagermachlup1981a,capitaineOnsagerMachlupFunctionalElliptic2000}.
Particularly relevant here is the Onsager-Machlup functionals for time-dependent elliptic operators treated in \cite{CP14}. Also related to the presented approach is the application of the Onsager-Machlup theory for semi-martingales on manifolds that are represented as development of Euclidean semi-martingales, see e.g. \cite{sommerAnisotropicallyWeightedNonholonomically2016,grongMostProbablePaths2022} where the Onsager-Machlup functional is applied on anti-developments to determine most probable paths for semi-martingales on manifolds.

\section{Most probable flows}
\label{sec:most_probable_flows}
The stochastic process in \eqref{ShapeSDE} has infinitesimal generator
$$L_t = \frac{1}{2} \sum_{r=1}^J \sigma_{r,t}^2 + u_t,$$
see \cite[Section~4.2]{kunitaStochasticFlowsStochastic1997} for details.
We assume that $L_t$ is a elliptic operator which is equivalent to assuming that $\spn\{ \sigma_r\}_{r=1}^J = \mathbb{R}^d$ at any point. Define a cometric
$$g^*_t = \sum_{r=1}^J \sigma_{r,t} \otimes \sigma_{r,t},$$
which will be a Riemannian cometric by our assumptions. Write $\Delta_{g_t}$ for the Laplace-Beltrami operator of the corresponding Riemannian metric $g_t$ and define a vector field~$z_t$ such that
\begin{equation} \label{Loperator} 
L_t = \frac{1}{2}\Delta_{g_t} + z_t.
\end{equation}
For details around the stochastic flow $\varphi_t$, see Remark~\ref{re:Details}. The result of \cite{CP14} then tells us that the Onsager-Machlup functional is of the form $\int_0^T H(t,\gamma(t),\dot \gamma(t)) \, dt$ with
\begin{align} \label{OMFunc1}
H(t, \gamma, \dot \gamma) & = \frac{1}{2} \|\dot \gamma_t - z_t (\gamma_t)\|_{g_t}^2 +  \frac{1}{2} (\DIV_{g_t} z_t)(\gamma_t) - \frac{1}{12} S_t(\gamma_t) + \frac{1}{4} \tr_{g_t} \dot g_t|_{\dot \gamma_t} \\ \nonumber
& =: \frac{1}{2}  \|\dot \gamma_t - z_t(\gamma_t) \|^2_{g_t}  + f_t(\gamma_t).
\end{align}
Here, $S_t$ is the scalar curvature of $g_t$. We now have the following result.

\begin{theorem}[Most probable flows] \label{th:DetDrift}
Let $\nabla^t$ be the Levi-Civita connection of $g_t$ and let $z_t$ be the vector field in \eqref{Loperator}. With slight abuse of notation, we define adjoint operators $\dot g_t^*, (\nabla^t \xi)^* : TM \to TM$ by
\begin{equation} \label{AbuseNotation} \left\langle \dot g_t^* \xi_1, \xi_2 \right\rangle_{g_t} =
\dot g_t(\xi_1, \xi_2), \qquad \left\langle (\nabla^t \xi)^* \xi_1, \xi_2 \right\rangle_{g_t} =\left\langle \nabla_{\xi_2} \xi, \xi_1 \right\rangle_{g_t},\end{equation}
for vector fields $\xi$, $\xi_1$, $\xi_2$.
Then $\varphi_t$ is a most probable flow if and only if it solves the equation
$$\dot \varphi_t = (w_t+z_t)(\varphi_t),$$
with $w_t$ satisfying
\begin{equation} \label{GlobalEq} \dot w_t  + \nabla_{w_t+z_t}^t w_t  + \dot g_t^* w_t + (\nabla^t z_t)^* w_t =\nabla^t f_t.\end{equation}
\end{theorem}
We emphasize that $\dot g_t^*$ denotes the metric adjoint of $\dot g_t$ and not the time-derivative of the cometric. We remark that the above result was computed in \cite{CP14} for the special case of $g_t$ being given by the Ricci flow and $z_t = 0$. Theorem~\ref{th:DetDrift} follows from defining $\gamma_t = \varphi_t(x)$ for an arbitrary $x$ and applying the lemma below. Write $\frac{D^t}{dt}$ for the covariant derivative along the curve $\gamma_t$ with respect to $\nabla^t$. Note that $\frac{D^t}{dt} w_t(\gamma_t) = (\dot w_t)(\gamma_t) + (\nabla_{\dot \gamma_t}^t w_t)(\gamma_t)$ for the restriction of a time-dependent vector field along a curve.

\begin{lemma}
For two given fixed points, let $\gamma$ be a minimizing curve of the functional $H$ in \eqref{OMFunc1} connecting these end points. Then $\gamma_t$ is a solution of
\begin{equation} \label{CurveEq} \frac{D^t}{dt} \dot \gamma_t   + \dot g_t^* \dot \gamma_t + (\nabla^t z_t)^* \dot \gamma_t = \dot z_t  + \nabla_{\dot \gamma_t}^t z_t + \dot g_t^* z+t+  (\nabla^t z_t)^*  z_t +\nabla^t f_t.\end{equation}
\end{lemma}

\begin{proof}
Let $\gamma_t^s$ be a variation of $\gamma_t$ with fixed end points. We recall that
$\frac{D^t}{ds} \dot \gamma_t^s = \frac{D^t}{dt} \partial_s \gamma_t^s $. We emphasize that $\frac{D^t}{ds}$ is a derivative with respect to $s$, and that the dependence on $t$ is through taking the covariant derivative with respect to $\nabla^t$. We introduce notation $\partial_s \gamma_t^s|_{s=0} = v_t$ and $\dot \gamma(t) = a_t+z_t(\gamma_t)$. Differentiating \eqref{OMFunc1}, we compute
\begin{align*}
    & \frac{d}{ds}\int_0^T H(t,\gamma^s,\dot \gamma^s)|_{s=0} \, dt |_{s=0} \\
    & = \int_0^T \left( \left\langle a_t , \frac{D^t}{ds} (\dot \gamma_t^s - z_t) |_{s=0} \right\rangle_{g_t} + \left\langle \nabla^t f, \partial_s \gamma_t^s |_{s=0} \right\rangle_{g_t} \right) dt \\
    & = \int_0^T \left( \left\langle a_t , \frac{D^t}{dt} v_t  \right\rangle_{g_t} - \left\langle a_t , \nabla_{v_t}^t  z_t \right\rangle_{g_t} + \left\langle \nabla^t f, v_t \right\rangle_{g_t} \right) dt.
\end{align*}
Using that
\begin{align*}
\frac{d}{dt} \left\langle a_t , v_t  \right\rangle_{g_t}= \langle \dot g_t^* a_t, v_t \rangle_{g_t} + \left\langle \frac{D^t}{dt} a_t ,  v_t  \right\rangle_{g_t} 
+ \left\langle a_t , \frac{D^t}{dt} v_t  \right\rangle_{g_t},
\end{align*}
we get
\begin{align*}
\frac{d}{ds} \int_0^T H(t,\gamma^s,\dot \gamma^s)|_{s=0} \, dt 
& = \int_0^T \left\langle v_t, - \frac{D^t}{dt} a_t - \dot g_t^* a_t - (\nabla^t z_t)^* a_t  +\nabla^t f_t \right\rangle_{g_t} \, dt.
\end{align*}
We see that for this expression to vanish for any variational vector field $v_t$ we must have
\begin{align} 
\label{aequation}
0 & = \frac{D^t}{dt} a_t + \dot g_t^* a_t + (\nabla^t z_t)^* a_t -\nabla^t f_t.
\end{align}
The result follows.
\end{proof}
\begin{remark}
Alternatively, the equations $\dot \gamma = a_t + z_t(\gamma_t)$ and \eqref{aequation} can be used to find the flow lines of the most probable flow.
\end{remark}

\begin{remark} \label{re:Details}
Up until now, we have not gone into details about the Kunita flow itself. For our classes of vector fields and for a fixed $x$, we can always find a unique maximal solution $\varphi_t(x)$ of \eqref{ShapeSDE} with initial condition $\varphi_0(x) = x$ up to its explosion time, see, e.g., \cite[Theorem 3.4.5]{kunitaStochasticFlowsStochastic1997}. We need to be more careful when it comes to concluding that the flow $x \mapsto \varphi_t(x)$ is a true diffeomorphism of $M$. For the time-homogeneous case when the vector fields $\sigma_r = \sigma_{r,t}$ and $u = u_t$ have no dependence on $t$, then $\varphi_t$ is almost surely a diffeomorphism if the diffusion does not explode in finite time, see \cite{kunita2006decomposition}, \cite[Theorem~5.1.1]{li2021stochastic}. Sufficient conditions for non-explosions is completeness and lower-bound for the Ricci curvature, or that it in general does not decay too fast when approaching infinity, see \cite{grigor1999analytic}, \cite[Theorem 5.1.1]{Hsu02} for details.
For the time-inhomogeneous case, a sufficient condition for the solution to be a diffeomorphism is for any $[0,T]$, there exists an $L^1$ function $C_t$ such that all of the vector fields $\sigma_{1,t}, \dots, \sigma_{J,t}, u_t$ have all spacial derivatives bounded by $C_t$, and furthermore, the vector fields themselves are bounded by $C_t(1+|x|)$, see \cite[Theorem~4.6.5]{kunitaStochasticFlowsStochastic1997}.
\end{remark}

\subsection{Local coordinate representations for numerical implementation}
We will below give explicit formulas for computing the most probable flow from the equation \eqref{ShapeSDE} that can be implemented numerically. We consider $M$ as a $\mathbb{R}^d$ with coordinates $x = (x^1, \dots, x^d)$. In what follows, all terms will depend on $t$, but we will suppress this dependence to simplify notation. Write $u = \sum_{k=1}^d u^k \partial_k$ and $\sigma_{r} = \sum_{i=1}^d \sigma^i_{r} \partial_i$.
\begin{enumerate}[\rm (1)]
\item The cometric $g^* = \sum_{i,j=1}^d g^{ij} \partial_i \otimes \partial_j$ is given by $g^{ij} = \sum_{r=1}^J \sigma_{r}^i \sigma_{r}^j$. The corresponding Laplacian is given by
$$\Delta_{g} = \sum_{i,j=1}^d g^{ij} \partial_i \partial_j + \sum_{i,j=1}^d \frac{1}{\sqrt{|g|}} \partial_i (\sqrt{|g|} g^{ij}) \partial_j.$$
Since
\begin{align*}
\sum_{r=1}^J \sigma_{r}^2 & = \sum_{r=1}^J \sum_{i,j=1}^d \sigma_{r}^i \sigma_{r}^j \partial_{i} \partial_j + \sum_{r=1}^J \sum_{i,j=1}^d (\sigma_{r}^i \partial_i \sigma_{r}^j) \partial_j \\
& = \Delta_{g} + \sum_{i,j=1}^d \left(\sum_{r=1}^J \sigma_{r}^j \partial_i \sigma_{r}^i - \frac{1}{2} g^{ij} \partial_i \log |g| - \partial_i g^{ij} \right) \partial_j ,
\end{align*}
we have that $z = \sum_{j=1}^d z^j \partial_j$ equals
$$z^j  = u^j + \frac{1}{2} \sum_{i,j=1}^d \left(\sum_{r=1}^J \sigma_{r}^j \partial_i \sigma_{r}^i - \frac{1}{2} g^{ij} \partial_i \log |g| - \partial_i g^{ij} \right) \partial_j .$$
\item As usual, the Christoffel symbols of the Levi-Civita connection is given by $\Gamma_{ij}^k = \frac{1}{2} \sum_{l=1}^d g^{kl} \left( \partial_j g_{li} + \partial_i g_{lj} - \partial_l g_{ij} \right)$, with the scalar curvature expressed as
$$S = \sum_{i,j,k=1}^d g^{ij} \left(\partial_k \Gamma_{ij}^k - \partial_j \Gamma_{ik}^k + \sum_{l=1}^d (\Gamma_{ij}^l \Gamma_{kl}^k - \Gamma_{ik}^l \Gamma_{jl}^k) \right).$$
\item For the maps introduced in \eqref{AbuseNotation}, we have local expression
\begin{align*}
\dot g^* \partial_j &= \sum_{k=1}^d \dot g_{jk} g^{ik} \partial_i, \\
(\nabla z)^* \partial_j & = \sum_{i,l=1}^d g_{lj} g^{ik} \left(\partial_k z^l + \sum_{m=1}^d z^m \Gamma_{km}^l\right) \partial_i.
\end{align*}
\item Finally, $f = f_t$ is given by
$$f = \frac{1}{2} \sum_{l,m=1}^d g^{lm} \left(\frac{1}{2} \dot g_{lm} +\partial_l z^m + \sum_{k=1}^d z^k\Gamma^m_{lk}\right) - \frac{1}{12} S.$$
\end{enumerate}
For the global equation, we can write \eqref{GlobalEq} as
\begin{align*} 0 & = \dot w^i + \sum_{j=1}^d (w^j +z^j) \partial_j w^i + \sum_{j,k=1}^d (w^k + z^k) w^j \Gamma_{kj}^i + \sum_{j,k=1}^d \dot g_{jk} g^{ik} w^j  \\
&  \qquad + \sum_{j,k,l=1}^d g_{jl} g^{ik} w^j \left(\partial_k z^l + \sum_{m=1}^d z^m \Gamma_{km}^l \right) - \sum_{j=1}^d g^{ij} \partial_j f\end{align*}
while the solution of \eqref{CurveEq} is
\begin{align*}
0 & =  \ddot x^i - \dot z^i +  \sum_{j,k=1}^d \dot x^k \dot x_t^j \Gamma_{t,kj}^i  - \sum_{k=1}^d \dot x^k \left(\partial_k z_t^i + \sum_{j=1}^d \Gamma_{kj}^i z^j \right) + \sum_{j,k=1}^d \dot g_{jk} g^{ik} (\dot x^j- z^j) \\
& \qquad+ \sum_{j,k,l=1}^d g_{lj} g^{ik} \left(\partial_k z^l + \sum_{,=1}^d z^m \Gamma_{km}^i \right) (\dot x^j - z^j) - \sum_{k=1}^d g^{ij} \partial_j f .
\end{align*}
Figure~\ref{fig:AC} visualizes numerical realizations of most probable flows acting on landmarks with $u = u_t$ solving the Euler-Poincar\'e equations \eqref{eq:EP}.

\begin{example}[Brownian background noise] \label{ex:Background}
Assume that we have that $J=d$ and $\sigma_r = \partial_r$. Then $g$ is the Euclidean metric and $S=0$. Furthermore $\sum_{r=1}^d \sigma_r^2 = \Delta_g$, so $z_t = u_t$ and $f_t =\frac{1}{2} \sum_{r=1}^d \partial_r u_t^r$.
 We will let $\bar{\nabla}$ denote the flat connection on $\mathbb{R}^d$, that is, $\bar{\nabla}_u w = (Dw)u$ where $Dw$ is the Jacobi matrix of $w$. The equation for the most probable flow is then $\dot\varphi_t = (w_t+ u_t)(\varphi_t)$ with $w_t$ being a solution of
$$ \dot w_t  + \bar{\nabla}_{w_t+u_t} w_t  + (\bar{\nabla} u_t)^* w_t =\frac{1}{2} \bar{\nabla} \DIV u_t,$$
or
$$\dot w_t^i  + \sum_{j=1}^d \left((w_t^j + u_t^j) \partial_j w^i_t  + w_t^j \partial_i u_t^j -  \frac{1}{2} \partial_i \partial_j u_t^j \right) =0,$$
A most probable path is a solution
$$\ddot x_t^i  + \sum_{j=1}^d \dot x_t^j  \partial_i u_t^j(x_t)  = \frac{1}{2} \partial_i (\|u_t\|^2 + \DIV(u_t))(x_t)=0.$$
\end{example}

\section{Perturbed right-invariant geodesic flows}
\label{sec:min_energy_flows}
We now focus on specific examples of Kunita SDEs, mainly arising from geometric mechanics and shape analysis. Right-invariant metrics on the diffeomorphism group $\Diff(M)$ determine the drift field $u_t$ as derivative of geodesic equations in the group. These are in turn described by the EPDiff equations (Euler-Poincar\'e on diffeomorphisms) and used extensively for example in the Large Deformation Diffeomorphic Metric Mapping (LDDMM) shape matching framework, see \cite{younesShapesDiffeomorphisms2010,holmGeometricMechanicsPart2011a}. Stochastic extensions of the EPDiff equations include the stochastic EPDiff equations \cite{holmVariationalPrinciplesStochastic2015} and the stochastic EPDiff perturbations \cite{ACC14}. In the former scheme, the minimization of a stochastic variational principle implies that the drift becomes stochastic and the flow thus non-Brownian. In the later, taking expected energy results in deterministic drift and thus Brownian flows. Below, we will outline these schemes and relate to most probable paths both theoretically and numerically. 


\subsection{Reduction and EPDiff equations}
Let $M$ be a domain in $\mathbb{R}^d$ and let $\Diff(M)$ be its group of diffeomorphisms. We consider $\Diff(M)$ as a Lie group with Lie algebra given by vector fields $\calX(M)$. With this formalism, $\ad(u)v = -[u,v]$ is the negative of the usual Lie algebra for $u,v \in \calX(M)$, while for $\varphi \in \Diff(M)$, we have $\Ad(\varphi)u (x) = \varphi_{*} u(\varphi^{-1}(x))$. If $\varphi_t$ a curve in $\Diff(M)$, then $\varphi_t$ has right logarithmic derivative
\begin{equation} \label{eq:flow}
    \delta^r \varphi_t =: \dot \varphi_t \cdot \varphi_t^{-1} =u_t \in \calX(M) \text{ if } \dot \varphi(x) = u_t(\varphi_t(x)). 
\end{equation}
Let $A$ be a positive, self-adjoint operator on $\calX(M)$. Define
\begin{equation}
\langle u, u \rangle_{\calX(M)}= \| u \|_{\calX(M)}^2 = \int_M \langle  Au, u \rangle dx, \qquad u\in \calX(M).
\label{eq:metric}
\end{equation}
This can be used to measure the energy of a curve $\varphi_t$ by
\begin{equation} \label{Energy} E(\varphi_t) = \frac{1}{2} \int_0^T \| \delta^r \varphi_t \|^2_{\calX(M)} dt.\end{equation}
With $\varphi_t$ satisfying \eqref{eq:flow},
a first-order condition for $\varphi_t$ to minimize the energy between an initial state $\varphi_0$ and a final state $\varphi_T$ is that it satisfies the Euler-Poincar\'e equation
\begin{equation}
\dot u_t + \ad(u_t)^* u_t =0,
\label{eq:EP}
\end{equation}
where $\ad(u_t)^*$ here is the adjoint of $\ad(u_t)w_t = - [u_t, w_t]$ with respect to the inner product $\langle \cdot , \cdot \rangle_{\calX(M)}$. Explicitly, 
$$\dot m_t + (Dm_t) u_t + (\DIV u_t) m_t + (Du_t)^Tm_t =0 , \qquad m_t = Au_t.$$
See e.g. \cite{HTY09} for more details.

The inner product $\left<u,v\right>_{\mathcal X(M)}=\left<Au,v\right>_{L^2}$ used in \eqref{eq:metric} defines a right-invariant Riemannian metric on the subgroup $G=\{\phi_T|\phi_0=\id,\ \|u\|<\infty\}$ of end points $\phi_T$ of finite energy paths satisfying the flow equation \eqref{eq:flow} by setting
\begin{equation*}
    \left<v,u\right>_{T_\phi G}
    =
    \left<v\circ\phi^{-1},u\circ\phi^{-1}\right>_{\mathcal X(M)}
    .
\end{equation*}
The Euler-Poincar\'e equations are geodesic equations for this metric \cite{arnoldGeometrieDifferentielleGroupes1966,arnoldMathematicalMethodsClassical1978}. 

\subsection{Stochastic Euler-Poincar\'e flows}
A stochastic version of the EPDiff equations can be obtained by perturbing the reconstruction equation \eqref{eq:flow} to obtain
\begin{equation}
    d\varphi_t(x) = u_t(\varphi_t(x))dt
    + \sum_{r=1}^J \sigma_r(\varphi_t(x)) \circ dW^r
    \label{eq:perturbed_rec}
\end{equation}
for $J$ vector fields $\sigma_r\in\mathcal X(M)$ are deterministic, but we allow $u_t$ to be random. This results in the stochastic flows introduced in~\cite{holmVariationalPrinciplesStochastic2015} for fluids and subsequently in~\cite{arnaudonGeometricFrameworkStochastic2019} for shapes. Because the stochastics is added to the right logarithmic derivative, much of the right-invariant geometric structure is preserved, for example the momentum map. This implies that critical flows for the energy \eqref{Energy} with the relation between $u_t$ and $\phi_t$ given by \eqref{eq:perturbed_rec} can identified with a reduction principle analogous to the deterministic case. Specifically, one has the following stochastic extension of the Euler-Poincar\'e equations \eqref{eq:EP}.
\begin{theorem}[\cite{holmVariationalPrinciplesStochastic2015,arnaudonGeometricFrameworkStochastic2019}]
    With the stochastic perturbed \eqref{eq:perturbed_rec}, critical flows for \eqref{Energy} take the form
\begin{equation}
\dot m_t 
+ \ad(u_t)^* m_t 
+ \sum_{r=1}^J\ad(\sigma_r)^* m_t \circ dW_t^r
=0,
\label{eq:stochasticEP}
\end{equation}
with momentum $m_t=Au_t$ and with $W_t$ an $\mathbb R^J$-valued Wiener process. 
\end{theorem}
Note that $u_t$ satisfying the system \eqref{eq:stochasticEP} will be stochastic, and the stochastic EPDiff flows are therefore inherently different from the most probable flows, while both arise from a variational principle.

\subsection{Energy minimizing flows}
An alternative approach has been presented \cite{ACC14} for the setting of Riemannian manifolds, which we consider here in the flat case. Consider a semi-martingale $\varphi_t$ in our open subset $M$ of $\mathbb{R}^d$ with induced filtration $\scrF_t$, and let $w_t = \int_0^t d\varphi_t \circ \varphi_t$ be its anti-development with respect to the flat connection, which is also a semi-martingale. We introduce a stochastic analogue of the left logarithmic derivative in \eqref{eq:flow},
\begin{equation*}  \textstyle \delta^r \varphi_t = \lim_{\ve \downarrow 0} \mathbb{E}\left[\frac{w_{t+\ve} -w_t}{\ve} |\mathscr{F}_t \right].\end{equation*}
We use this definition to introduce the expected energy
\begin{equation} \label{ExpEnergy} \textstyle E[\varphi_{\cdot}] = \frac{1}{2} \mathbb{E} \left[ \int_0^T \| \delta^r \varphi_t \|_{\calX(M)} dt \right],\end{equation}
with the expectation taken with respect to the measure induced by $\varphi_{[0,T]}$ on continuous paths from $[0,T]$ into $M$.
Furthermore, for any non-random vector field $v_t$ with $v_0 = v_T =0$, we define $\psi_{t}^s$ as the solution of $\psi_0^s = \id$, $\delta^r \psi^s_t = s \dot v_t$. We say that $\varphi_t$ is a critical point of the expected energy $E$ if $\frac{d}{ds} \mathbb{E}[\psi^s_{\cdot}  \circ \varphi_{\cdot} ]|_{s=0}=0$
for any such vector field $v_t$.

Consider now the case when $\varphi_t$ is the solution of \eqref{ShapeSDE},
giving us that $w_t = \int_0^t (u_t dt + \sum_{r=1}^J \sigma_{r,t} \circ dW^r)$ and that 
\begin{equation} \label{StockDelr} \delta^r \varphi_t = \lim_{\ve \downarrow 0} \mathbb{E}\left[\frac{w_{t+\ve} -w_t}{\ve} |\mathscr{F}_t \right] = u_t + \frac{1}{2} \sum_{r=1}^J \bar{\nabla}_{\sigma_{r,t}}\sigma_{r,t}=: \hat u_t,\end{equation}
is deterministic. We then have the following result.

\begin{theorem} \label{th:OptimalDrift}
For the flat connection $\bar{\nabla}$ on $\calX(M)$, define its Hessian by $\bar{\nabla}^2_{u,v} w = \bar{\nabla}_u \bar{\nabla}_v w - \bar{\nabla}_{\bar{\nabla}_u v} w$. Introduce a second order operator on vector fields by
$$\Box^\sigma_t = \sum_{r=1}^J \bar{\nabla}^2_{\sigma_{r,t}\sigma_{r,t}}.$$
Then $\varphi_t$ is critical if and only if $\hat u_t = u_t + \frac{1}{2}\sum_{r=1}^J \bar{\nabla}_{\sigma_{r,t}}^2 \sigma_{r,t}$ satisfies
\begin{equation} \label{OptU}
\dot {\hat u}_t+ \frac{1}{2} \Box^\sigma_t \hat u_t +  \ad(\hat u_t)^* {\hat u}_t = 0
\end{equation}
where $\ad(\hat u_t)^*$ is the dual with respect to $\langle \cdot , \cdot \rangle_{\calX(M)}$.
\end{theorem}
We remark that if we reformulate the Stratonovich SDE to the It\^o SDE \eqref{Ito}, then $\hat u_t$ in Theorem~\ref{th:OptimalDrift} is exactly the drift term in the It\^o SDE. Furthermore, if we consider $\Box^\sigma_t$ acting on each component, then $\Box^\sigma_t v = \sum_{i=1}^d \sum_{j=1}^J (\sigma_{j,t}^2 - \bar{\nabla}_{\sigma_{j,t}} \sigma_{j,t}) v_j \partial_{j}$ and $\sum_{j=1}^J (\sigma_{j,t}^2 - \bar{\nabla}_{\sigma_{j,t}} \sigma_{j,t})$ is the generator of the drift term in~\eqref{Ito}.

The proof of Theorem~\ref{th:OptimalDrift} can be deduced from that of \cite[Theorem~3.2 and Theorem~3.4]{ACC14}, but we include it here for the sake of completeness.
\begin{proof}
Before we begin the proof, let us consider the following property of the flat connection $\bar{\nabla}$. Recall that the Hessian is the operator $\bar{\nabla}^2_{u,v} = \bar{\nabla}_{u} \bar{\nabla}_v - \bar{\nabla}_{\bar{\nabla}_u v}$. Since $\bar{\nabla}$ is flat and torsion free, its Hessian $\bar{\nabla}^2$ is symmetric. We then see that
\begin{align*}
& [u, \bar{\nabla}_v w] - \bar{\nabla}_{[u,v]} w - \bar{\nabla}_{v} [u,w] \\
& = \bar{\nabla}_u \bar{\nabla}_v w - \bar{\nabla}_{\bar{\nabla}_v w} u - \bar{\nabla}_{\bar{\nabla}_u v} w + \bar{\nabla}_{\bar{\nabla}_v u} w - \bar{\nabla}_v \bar{\nabla}_u w + \bar{\nabla}_v \bar{\nabla}_w u \\
& = \bar{\nabla}^2_{u,v} w - \bar{\nabla}_{v,u}^2 w + \bar{\nabla}_{v,w}^2 u = \bar{\nabla}_{v,w}^2 u.
\end{align*}

Write $\varphi^s_t = \psi_t^s \circ \varphi_t$. By the chain rule for Stratonovich differentials,
\begin{align*} d \varphi_t^s & = \dot \psi_t^s(\varphi_t) + \psi_{t,*}^s \circ dw_t(\varphi_t) = \dot \psi_t^s((\psi_t^s)^{-1} \circ \varphi_t^s) + \psi_{t,*}^s \circ dw_t(( \psi_t^s)^{-1} \circ \varphi_t)  \\
& =  s \dot v_t (\varphi_t^s) + \Ad(\psi_t^s)u_t (\varphi_t^s) dt + \sum_{r=1}^J (\Ad(\psi_t^s)\sigma_{r,t})(\varphi_{t}^s) \circ dW^r,\end{align*}
and so by \eqref{StockDelr},
$$\delta^r \varphi_t^s = s\dot v_t + \Ad(\varphi_t^s) u_t + \frac{1}{2} \sum_{r=1}^J \bar{\nabla}_{\Ad(\psi_t^s)\sigma_{r,t}}\Ad(\psi_t^s) \sigma_{r,t},$$
Furthermore,
$$\tfrac{d}{ds}\delta^r \varphi_t^s |_{s=0} = \dot v_t + \ad(v_t) u_t + \frac{1}{2} \sum_{r=1}^J \bar{\nabla}_{\ad(v_t) \sigma_{r,t}}\sigma_{r,t}  + \frac{1}{2} \sum_{r=1}^J \bar{\nabla}_{ \sigma_{r,t}} \ad(v_t)\sigma_{r,t}.$$
Using our observation for $\bar{\nabla}$, we have
\begin{align*}
\tfrac{d}{ds}\delta^r \varphi_t^s |_{s=0} & = \dot v_t + \ad(v_t) \hat u_t - \frac{1}{2} \sum_{r=1}^J \bar{\nabla}_{\sigma_{r,t}, \sigma_{r,t}}^2 v_t \\
& = \dot v_t - (\ad(\hat u_t) + \frac{1}{2} \Box^\sigma_t) v_t .
\end{align*}
We then see that
\begin{align*}
\frac{d}{ds}E[\varphi_t^s]|_{s=0} &= 
 \int_0^T \left\langle \hat u_t , \dot v_t - (\ad(\hat u_t) + \frac{1}{2} \Box^\sigma_t) v_t \right\rangle_{\calX(M)} dt  
\end{align*}
and using that $\Box^\sigma_t$ is a symmetric operator with respect to $\langle \cdot , \cdot \rangle_{\calX(M)}$, the result follows.
\end{proof}

\begin{remark}
For context of how we used the properties of the flat connection in the proof of Theorem~\ref{th:OptimalDrift}, observe that for a general connection $\nabla$ with curvature $R$ and torsion $T$, 
\begin{align*}
& [u, \nabla_v w] - \nabla_{[u,v]} w - \nabla_{v} [u,w] \\
& = \nabla_{v,w}^2 + R(u,v)w + u+ (\nabla_v T)(u, w) + T(\nabla_v u, w).
\end{align*}
\end{remark}

\subsection{Most probable flows}
Let $\phi_t$ and $u_t=\delta^r\phi_t$ satisfy the Euler-Poincar\'e equations \eqref{eq:EP}. We can add noise to the system with the stochastic reconstruction~\eqref{eq:perturbed_rec} and obtain the stochastic system~\eqref{eq:stochasticEP}. Alternatively, we can consider the energy minimizing flows \eqref{OptU} to recover a deterministic flow. The most probable flow equations of section~\ref{sec:most_probable_flows} provide a second way to recover a deterministic flow as being most probable instead of energy minimizing, by solving \eqref{GlobalEq}. 
\begin{figure}[!ht]
    \centering
    \begin{subfigure}
        \centering
        \includegraphics[width=.23\linewidth,clip=true,trim=100 50 100 50]{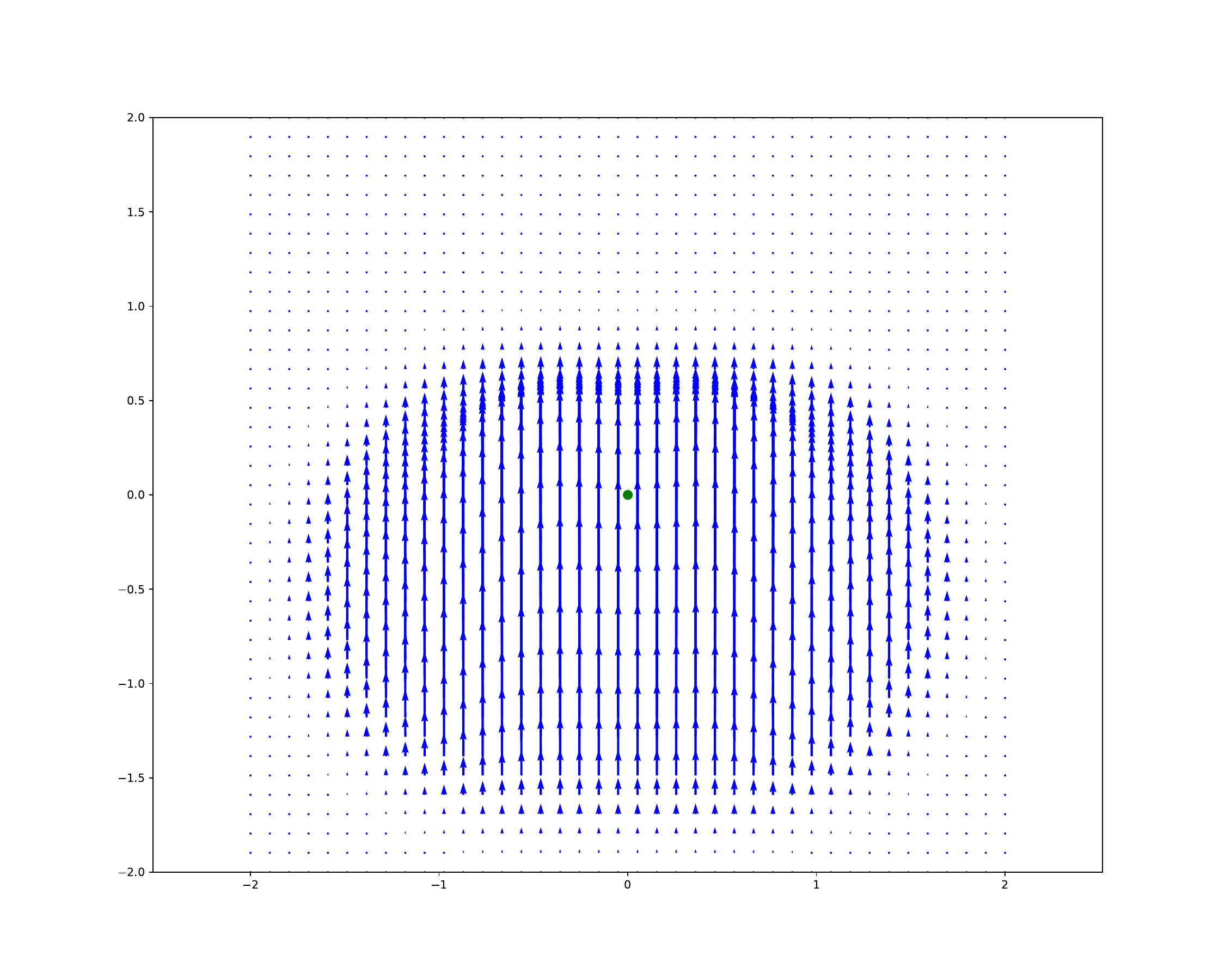}
    \end{subfigure}
    \hfill
    \begin{subfigure}
        \centering
        \includegraphics[width=.23\linewidth,clip=true,trim=100 50 100 50]{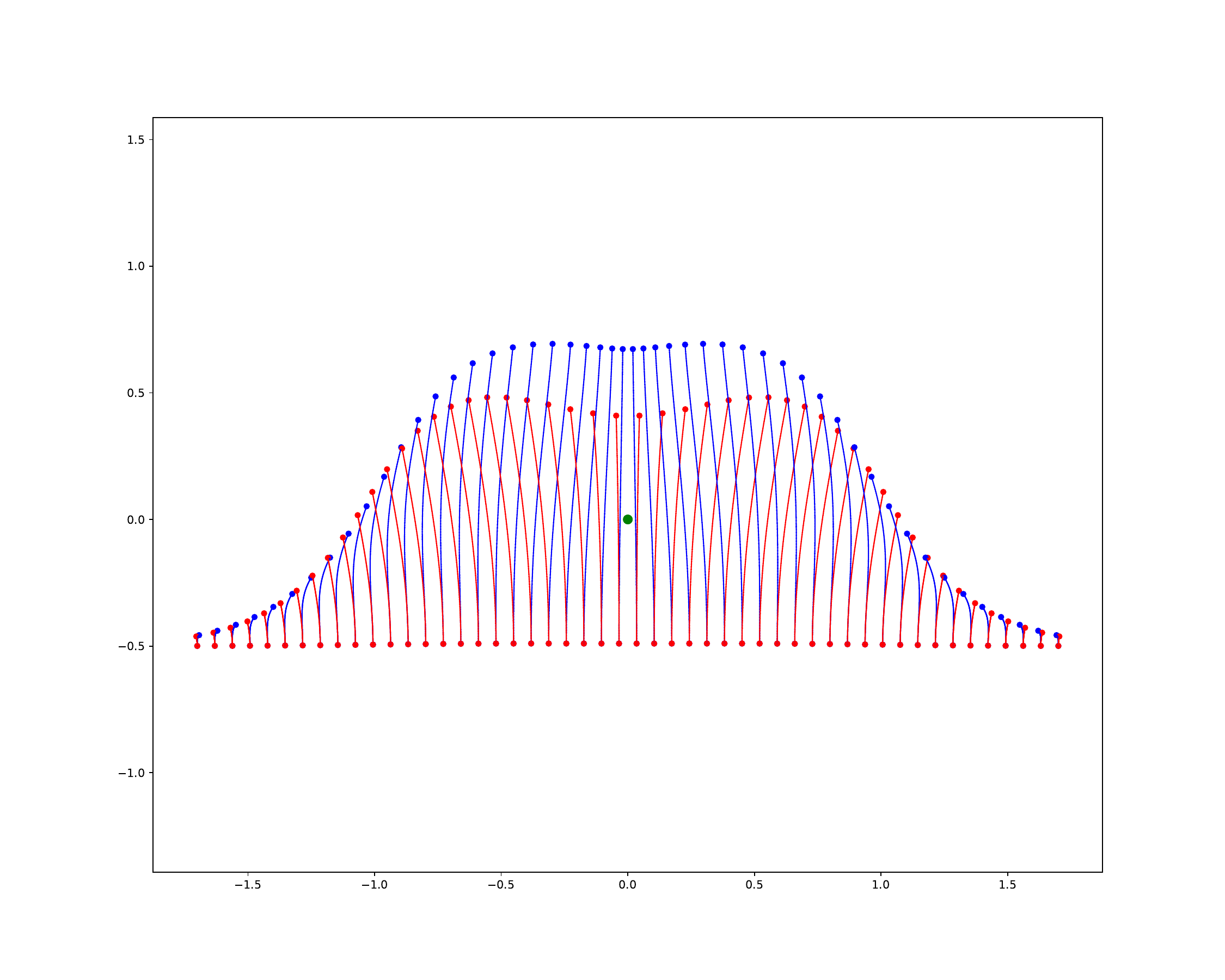}
    \end{subfigure}
    \hfill
    \begin{subfigure}
        \centering
        \includegraphics[width=.23\linewidth,clip=true,trim=100 50 100 50]{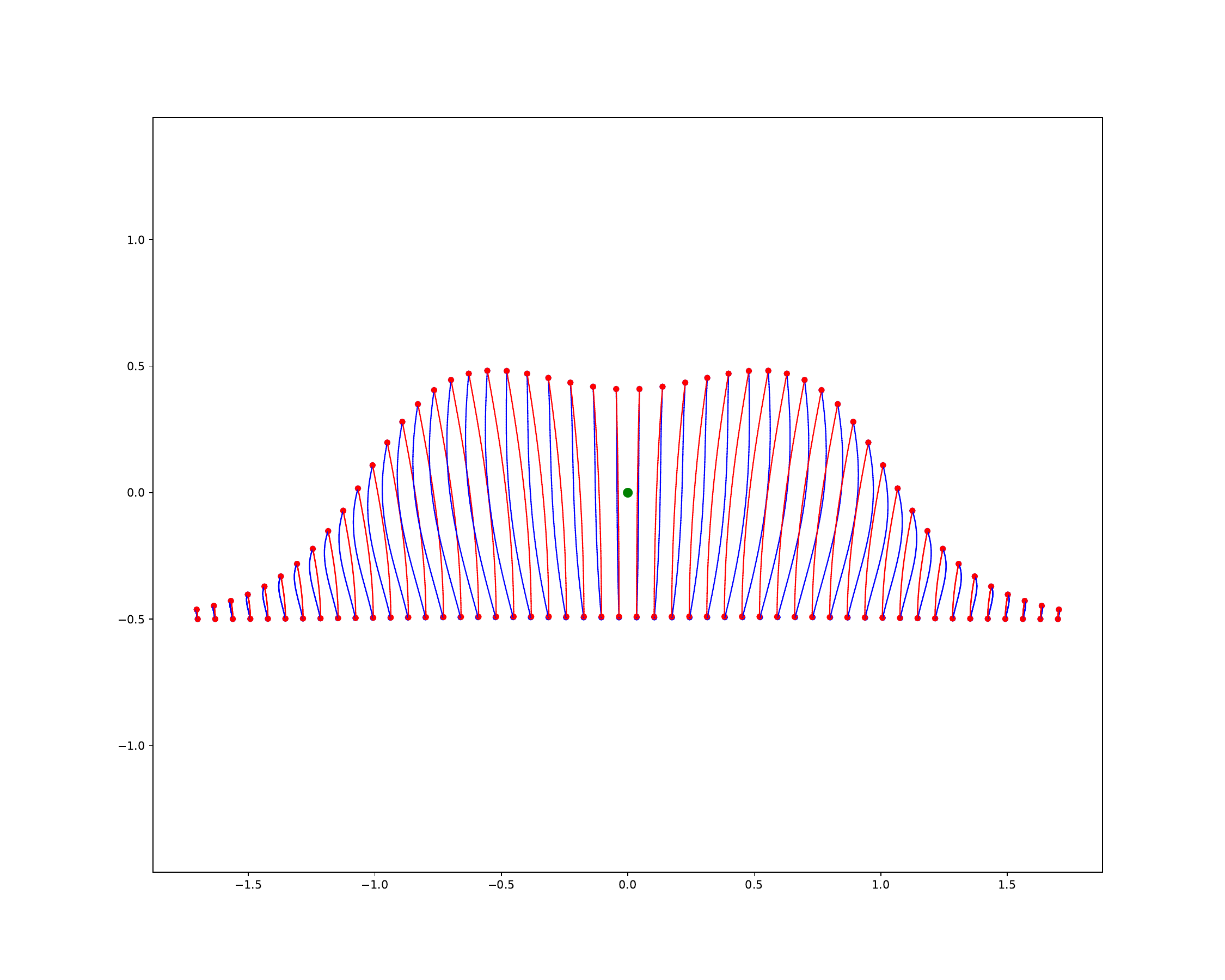}
    \end{subfigure}
    \begin{subfigure}
        \centering
        \includegraphics[width=.23\linewidth,clip=true,trim=100 50 100 50]{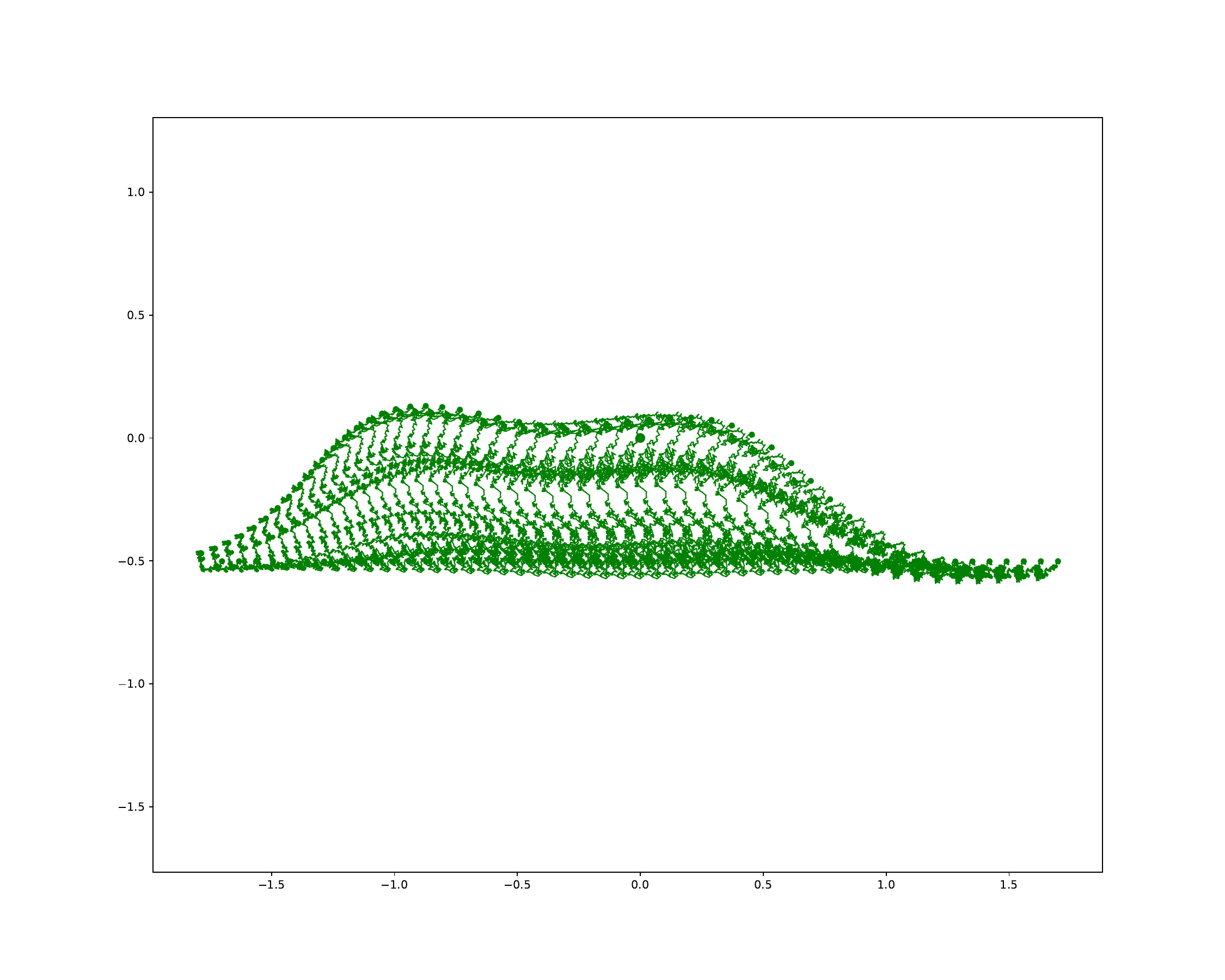}
    \end{subfigure}
    \begin{subfigure}
        \centering
        \includegraphics[width=.23\linewidth,clip=true,trim=100 50 100 50]{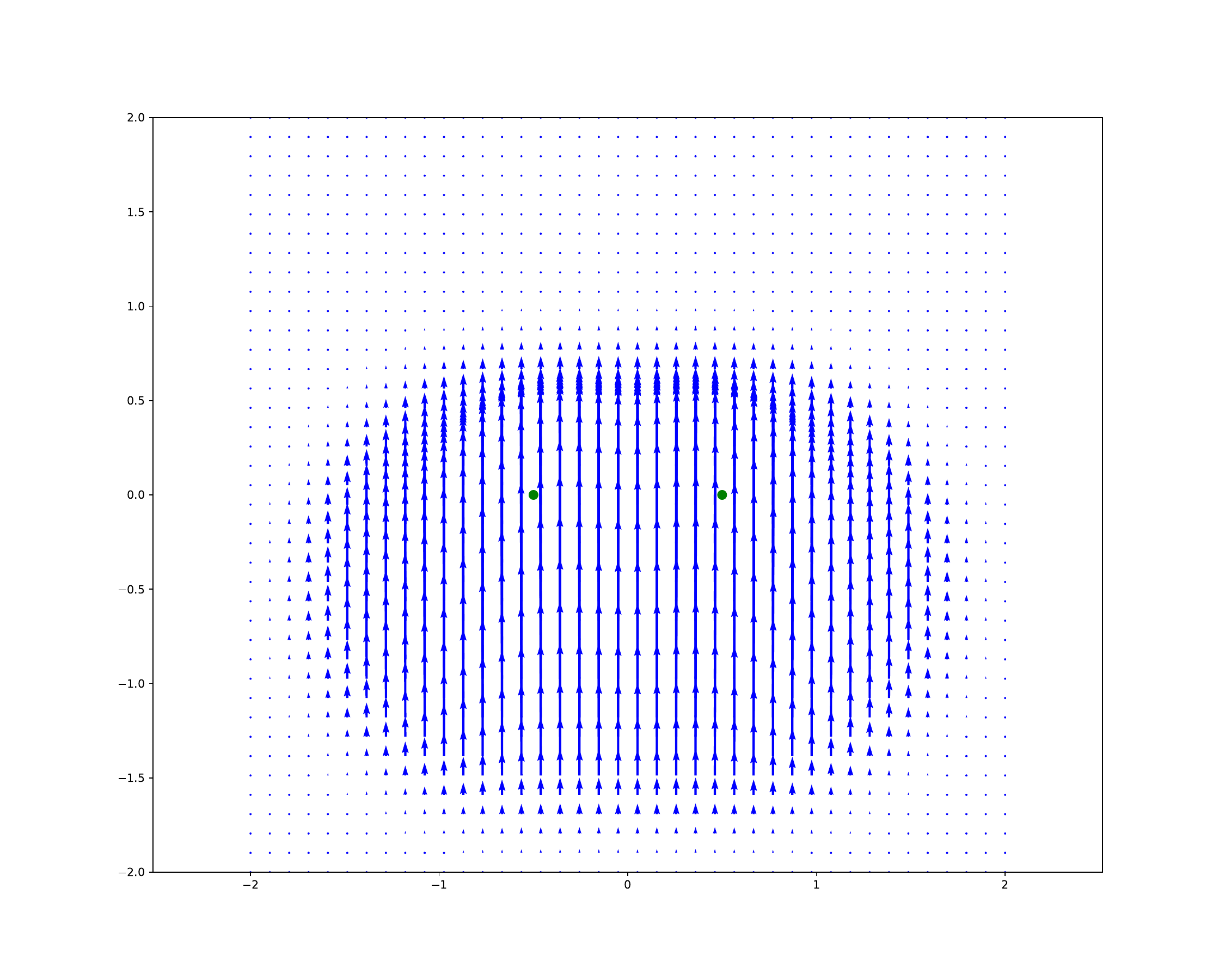}
    \end{subfigure}
    \hfill
    \begin{subfigure}
        \centering
        \includegraphics[width=.23\linewidth,clip=true,trim=100 50 100 50]{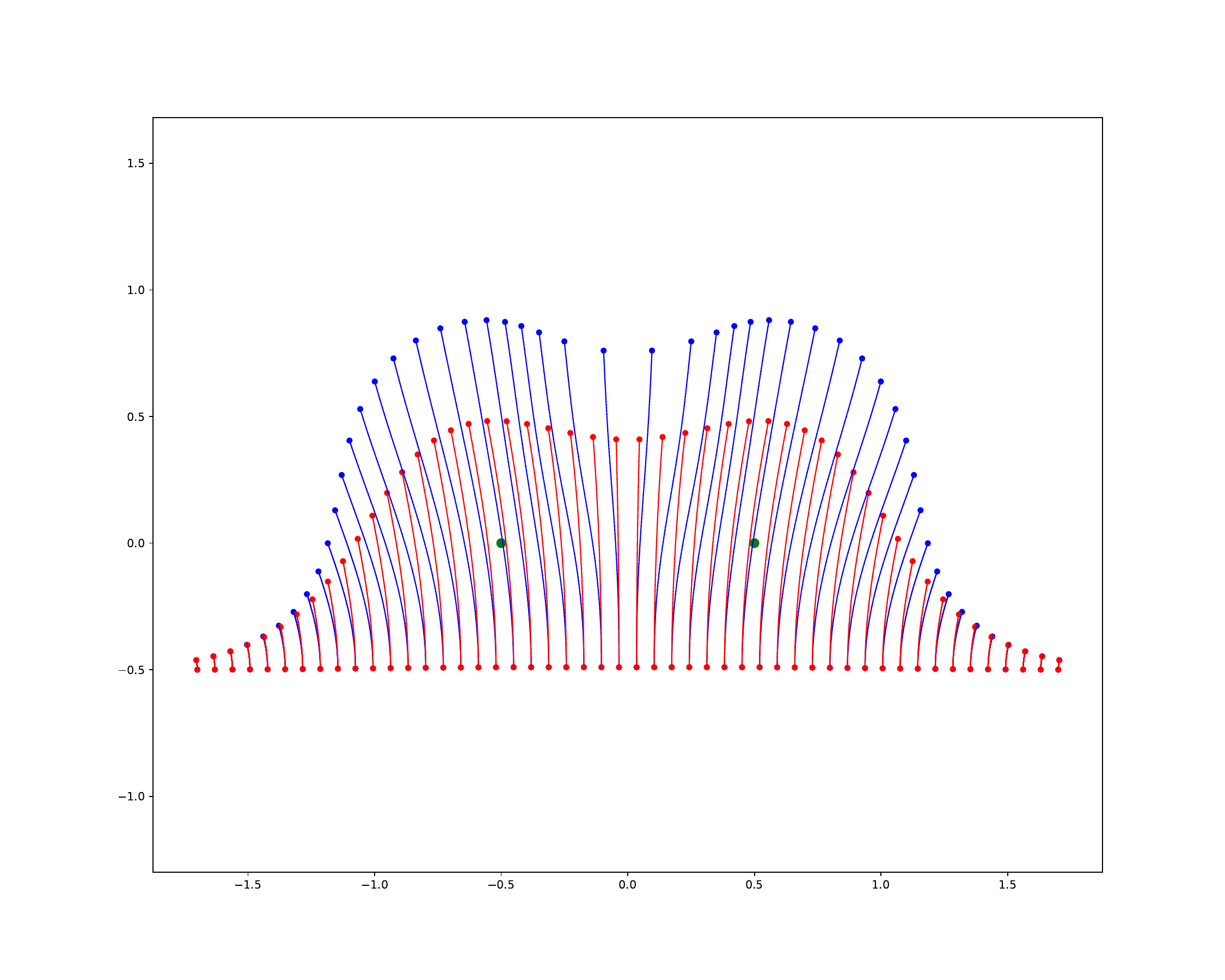}
    \end{subfigure}
    \hfil
    \begin{subfigure}
        \centering
        \includegraphics[width=.23\linewidth,clip=true,trim=100 50 100 50]{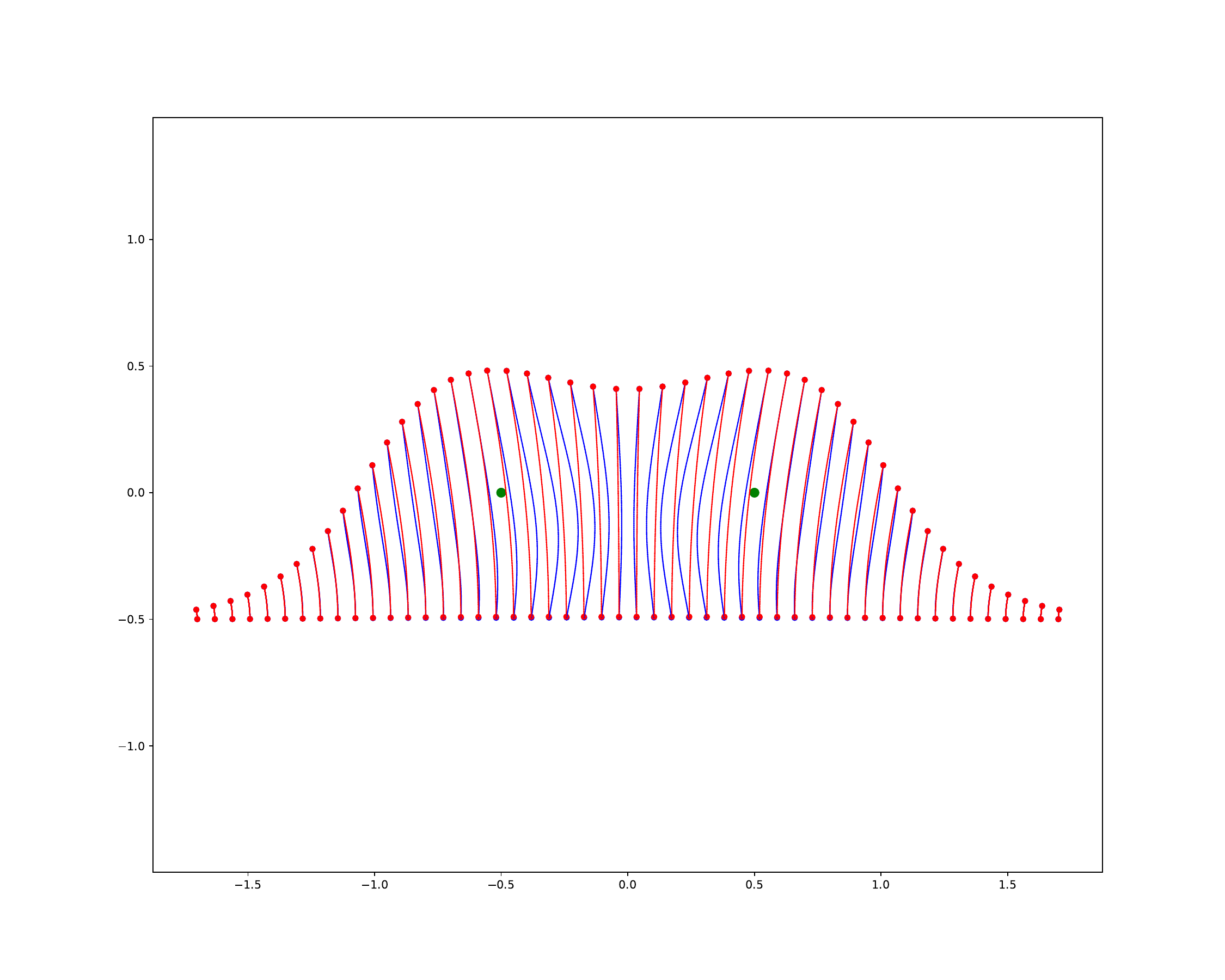}
    \end{subfigure}
    \begin{subfigure}
        \centering
        \includegraphics[width=.23\linewidth,clip=true,trim=100 50 100 50]{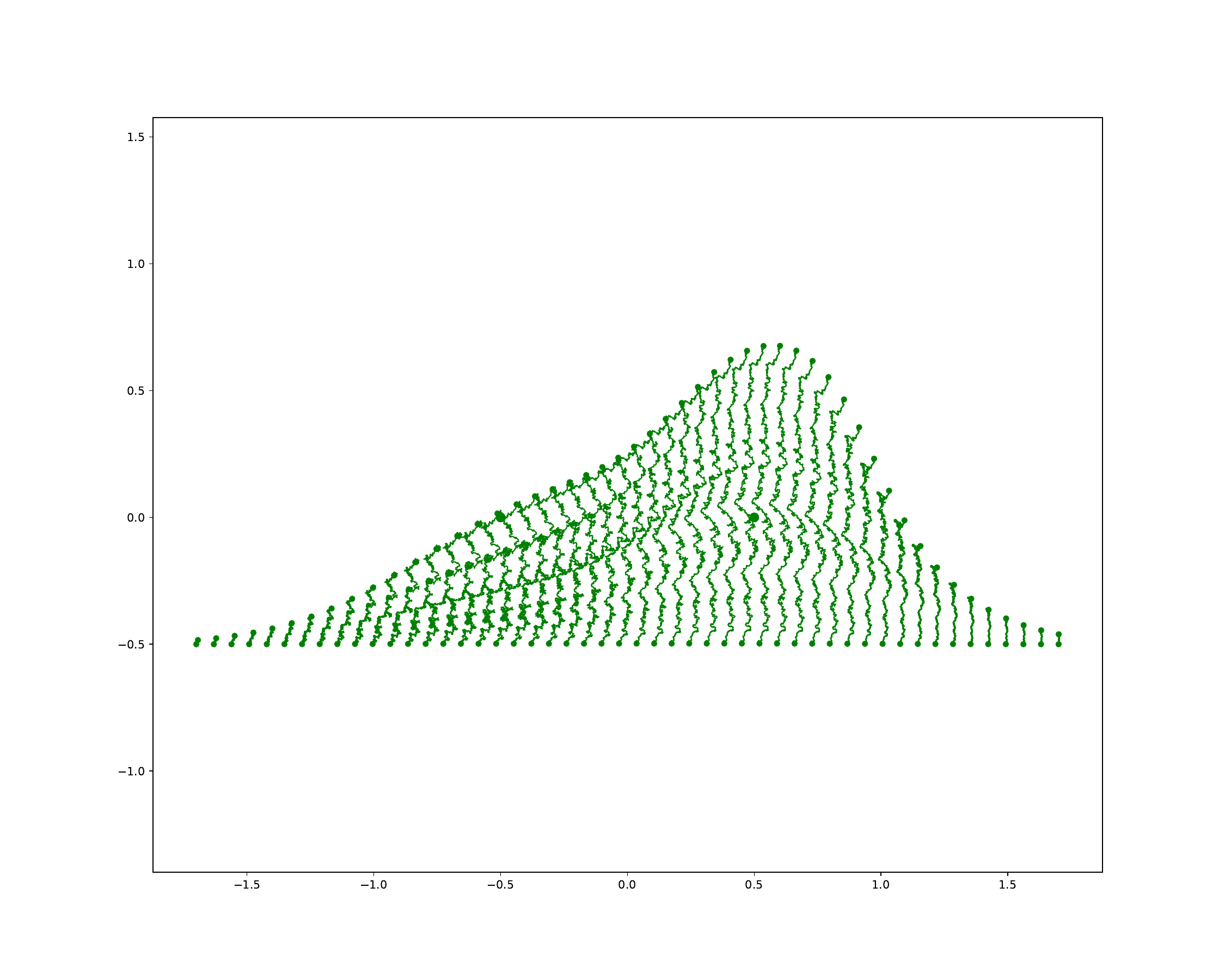}
    \end{subfigure}
    \begin{subfigure}
        \centering
        \includegraphics[width=.23\linewidth,clip=true,trim=100 50 100 50]{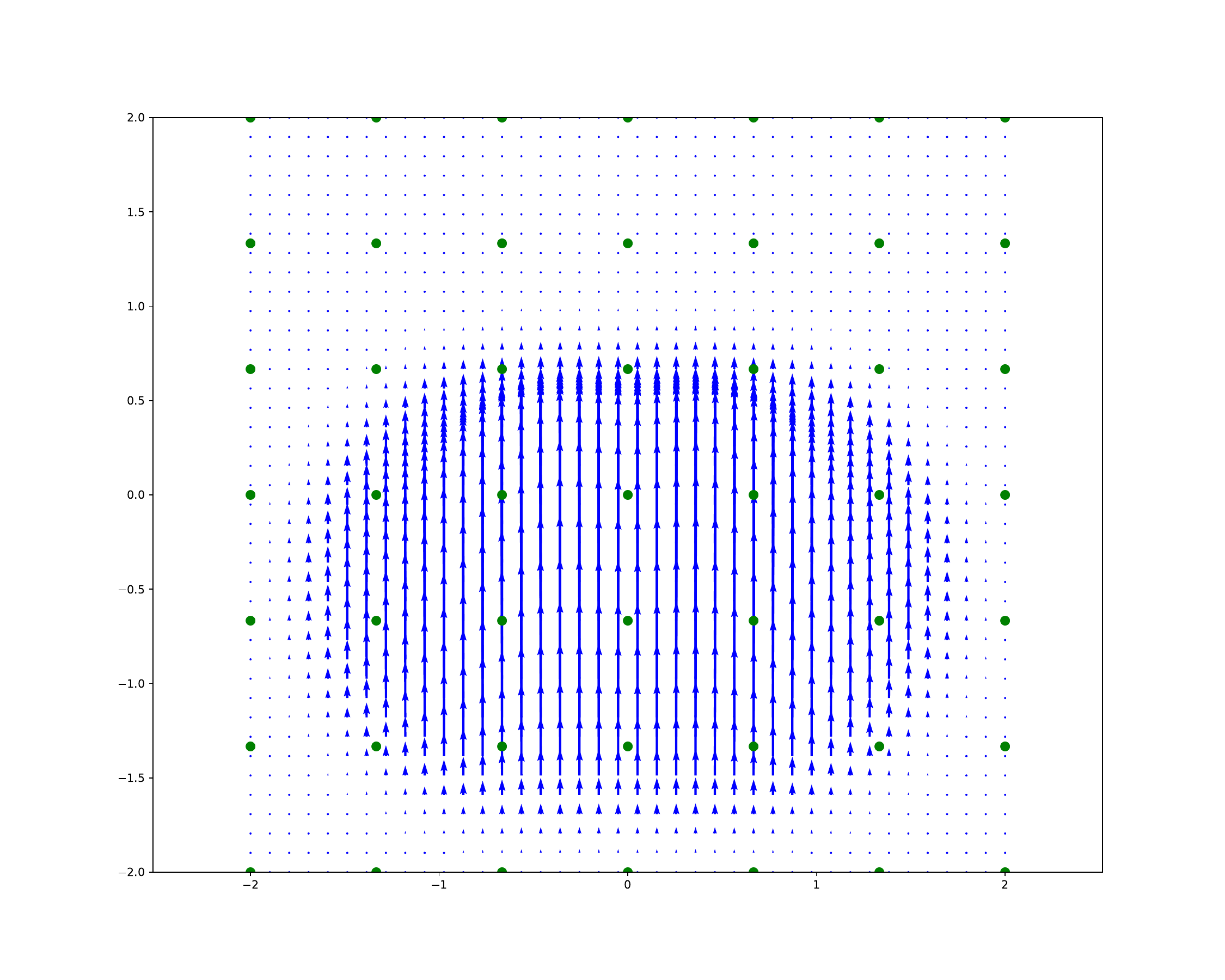}
    \end{subfigure}
    \hfill
    \begin{subfigure}
        \centering
        \includegraphics[width=.23\linewidth,clip=true,trim=100 50 100 50]{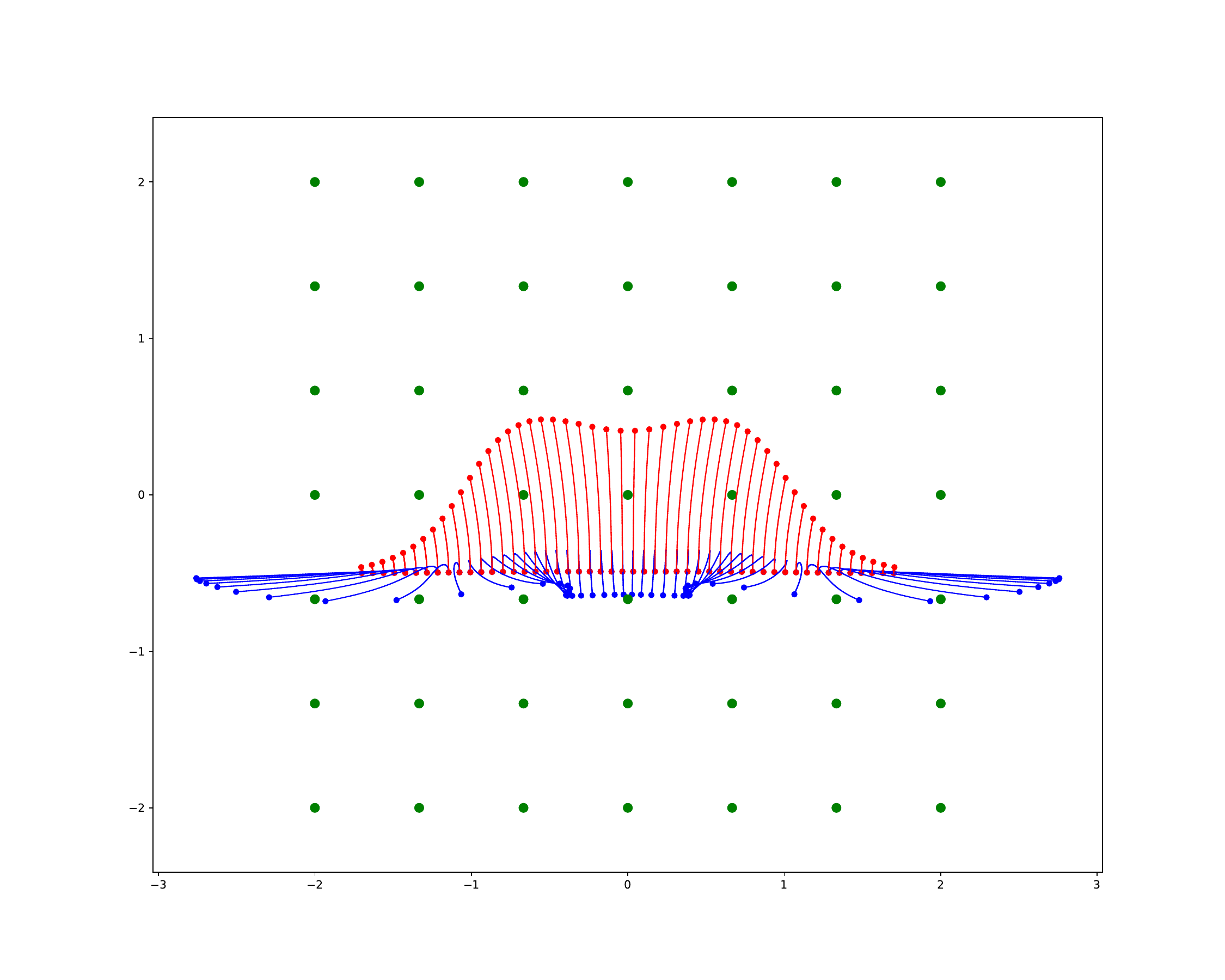}
    \end{subfigure}
    \hfill
    \begin{subfigure}
        \centering
        \includegraphics[width=.23\linewidth,clip=true,trim=100 50 100 50]{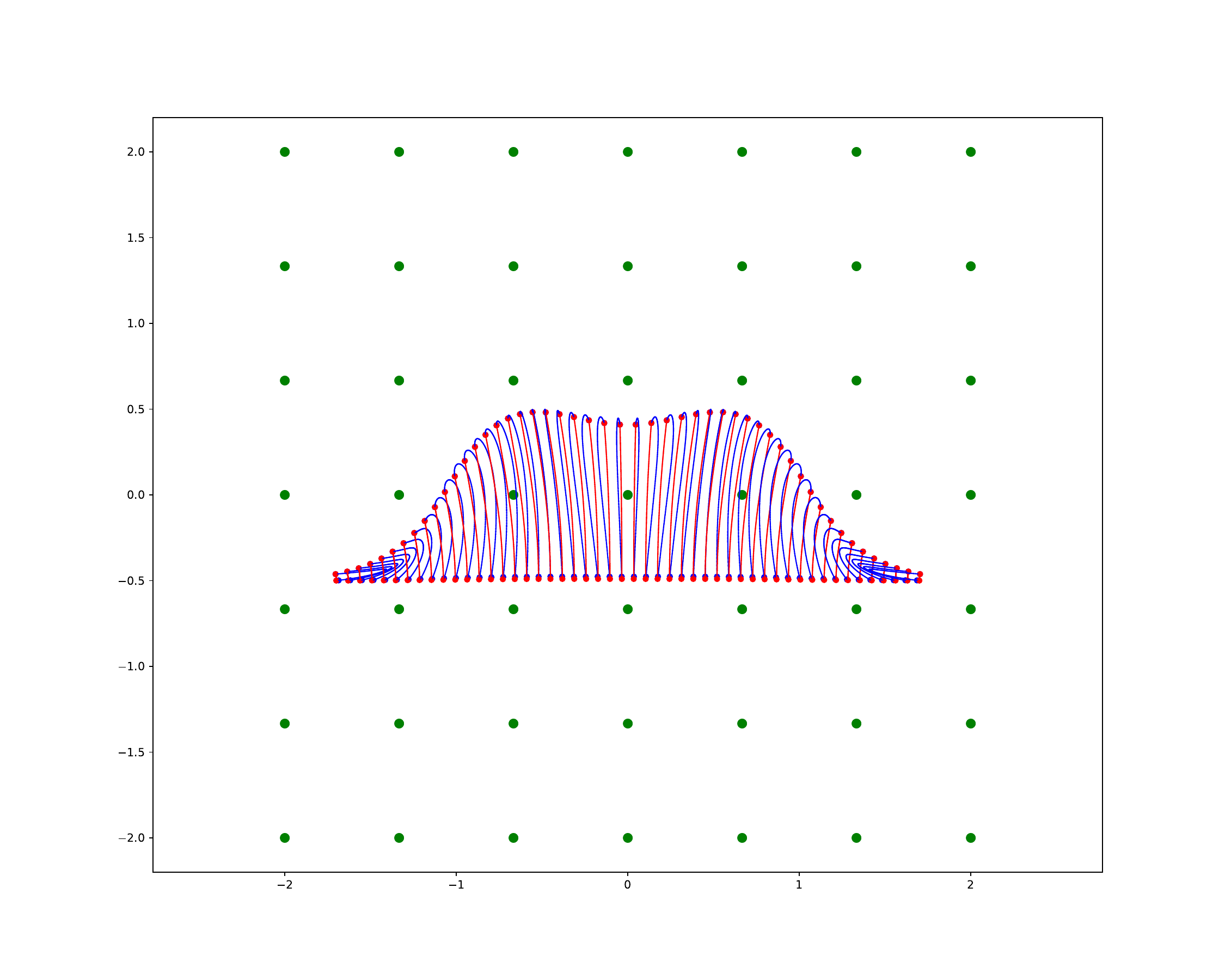}
    \end{subfigure}
    \begin{subfigure}
        \centering
        \includegraphics[width=.23\linewidth,clip=true,trim=100 50 100 50]{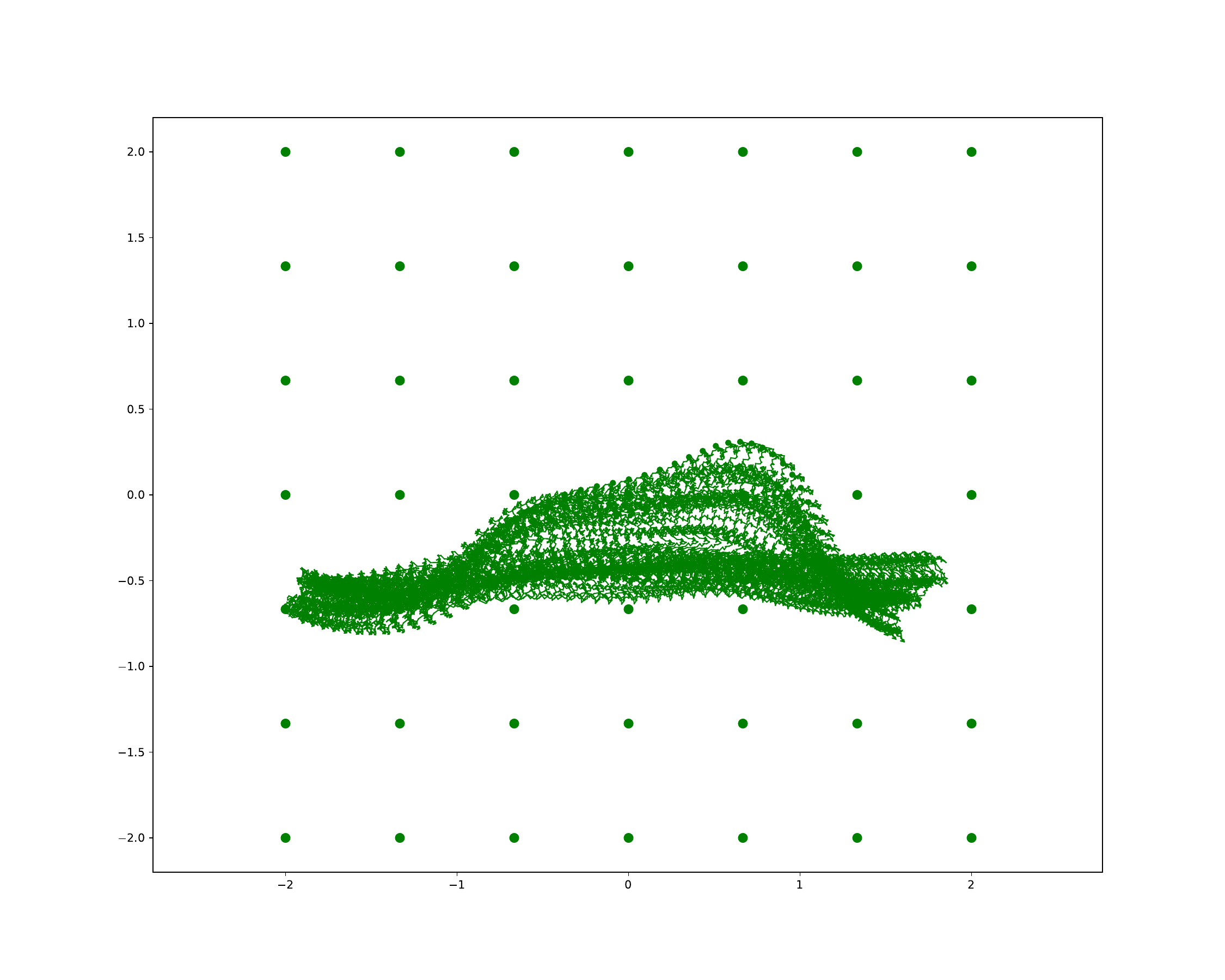}
    \end{subfigure}
    \caption{Red curves: deterministic flow (zero noise) with $u_t$ solving \eqref{OptU}; blue curves: MPP trajectories; green curves: stochastic trajectories. Left column: flow field $u_t$ at $t=0$ with noise centers (green points); center left column: forward integration of MPP equations for 40 landmarks evenly distributed at horizontal $y=-0.5$ line; center right column: initial value problem solved for each landmark between the end points of the deterministic landmark trajectories; right column: stochastic EPDiff sample path (noise amplitude downscaled for visualization). Top row: a single noise field at $(0,0)$; middle row: two noise fields at $(-0.5,0)$ and $(0.5,0)$; bottom row: grid of $7^2$ noise fields.}
    \label{fig:AC}
\end{figure}

\subsection{Visualization of perturbed flows}
We now numerically compare examples of solutions to the Euler-Poincar\'e equations, to the stochastic Euler-Poincar\'e equations, and most probable flows\footnote{The systems are implemented in Jax Geometry \url{https://bitbucket.org/stefansommer/jaxgeometry/}.}. Figure~\ref{fig:AC} shows forward flows of the unperturbed system satisfying the Euler-Poincar\'e equations. We then add noise in three different configurations ($1$, $2$, and $7^2$ Gaussian noise fields) and compute forward MPP equations, and solve the boundary value problem of most probable paths between the start and end points of the deterministic flow. The effect of the noise and resulting deviation of the most probable paths from the unperturbed trajectories is clearly visible in both cases. Lastly, we plot sample paths from the stochastic Euler-Poincar\'e equations. The landmark trajectories are highly correlated because of the spatial regularity of the noise fields, and the noise clearly perturbs the large scale dynamics of the system.

\bibliographystyle{abbrv}
\bibliography{Bibliography}

\end{document}